\documentclass[12pt,oneside]{amsart}
\usepackage{amsmath}
\usepackage{amsthm}
\usepackage{amsfonts}
\usepackage{amssymb}
\usepackage{enumerate}
\newtheorem{theorem}{Theorem}

\newtheorem{proposition}[theorem]{Proposition}
\newtheorem{corollary}[theorem]{Corollary}

\theoremstyle{definition}

\newtheorem{example}[theorem]{Example}

\theoremstyle{remark}
\newtheorem{remark}[theorem]{Remark}
\renewcommand{\phi}{\varphi}


\begin{document}

\title{$L_h^2$-functions in unbounded balanced domains}

\address{Carl von Ossietzky Universitat Oldenburg, Institut fur Mathematik, Postfach 2503, ¨
D-26111 Oldenburg, Germany}

\address{Institute of Mathematics, Faculty of Mathematics and Computer Science, Jagiellonian
University,  \L ojasiewicza 6, 30-348 Krak\'ow, Poland}

\author{Peter Pflug}\email{Peter.Pflug@uni-oldenburg.de}
\author{W\l odzimierz Zwonek}\email{wlodzimierz.zwonek@im.uj.edu.pl}
\thanks{The paper was written while second's author research stay at the Carl von Ossietzky University of Oldenburg supported by the Alexander von Humboldt Foundation. The second Author was also supported by the OPUS grant no. 2015/17/B/ST1/00996 of the National Science Centre, Poland. }


\begin{abstract} We investigate problems related with the existence of square integrable holomorphic functions on (unbounded) balanced domains. In particular, we solve the problem of Wiegerinck for balanced domains in dimension two. We also give a description of $L_h^2$-domains of holomorphy in the class of balanced domains and present a purely algebraic criterion for homogeneous polynomials to be square integrable in a pseudoconvex balanced domain in $\mathbb C^2$. This allows easily to decide which pseudoconvex balanced domain in $\mathbb C^2$ has a positive Bergman kernel and which admits the Bergman metric.
\end{abstract}
\maketitle

\section{Introduction}
As a starting point for the considerations presented in the paper the following problem of Wiegerinck from 1984 (see \cite{Wie 1984}) may serve. Does the Bergman space $L_h^2(D):=L^2(D)\cap\mathcal O(D)$ allow its dimension to be only either $0$ or $\infty$ for an arbitrary pseudoconvex domain $D\subset\mathbb C^n$? Wiegerinck confirmed the dichotomy in dimension one (see also \cite{Skw 1982}) and gave an example of an unbounded non-pseudoconvex domain $D\subset \mathbb C^2$ with  non-trivial $L_h^2$-functions but having its space $L_h^2(D)$ finitely dimensional. One of the main results of the present paper is Theorem~\ref{thm:description-dimension-two} in which the conjecture is confirmed for the class of balanced domains in dimension two. Moreover, in the same class of domains we present a complete description of domains admitting no non-trivial $L_h^2$-functions answering partially a question posed in \cite{Jar-Pfl 1987} and \cite{Sic 1985} about the characterization of such domains in arbitrary dimension. Let us also mention that a recent result on the infinite dimensionality of $L_h^2(D)$ (see \cite{Gal-Harz-Her 2016}) in a wide class of domains cannot be applied in the general case in our situation because the so called core of a domain may be equal to the whole domain  (belonging to the Siciak class of domains - definition below, see also the discussion in Section~\ref{subsection:trivial-dimension-two}).

Basic objects of study in several complex variables are domains of holomorphy defined as follows. For the family $\mathcal F\subset\mathcal O(D)$ we say that the domain $D\subset\mathbb C^n$ is an {\it $\mathcal F$-domain of holomorphy} if
there are no domains $D_0,D_1\subset\mathbb C^n$ with $\emptyset\neq D_0\subset D_1\cap D$, $D_1\subsetneq D$ such that
for any $f\in \mathcal F$ there exists an $\tilde f\in \mathcal O(D_1)$ with $\tilde f \equiv f$ on $D_0$.

Imitating the description of balanced $H^{\infty}$-domains of holomorphy (see e.\,g. \cite{Sic 1985}) we present an analogous description for balanced $L_h^2$-domains of holomorphy (in arbitrary dimension). The characterization given in Theorem~\ref{thm:main}  is expressed with the tools of extremal functions defined with the aid of homogeneous polynomials. As a consequence the description delivers some kind of an approximation of homogeneous logarithmically plurisubharmonic functions with the help of logarithms of roots of absolute values of homogeneous polynomials. That kind of approximation extends to some degree classical results in more general or similar situations (see the final discussion in  Section~\ref{subsection:domains-of-holomorphy}, \cite{Sib 1975}, \cite{Sic 1985}).

 The proof of Theorem~\ref{thm:description-dimension-two} follows from results of Jucha (see \cite{Juc 2012}) where the problem of Wiegerinck was studied for Hartogs domains. Employing the potential theoretic methods applied by Jucha we present additionally an effective and simple algebraic criterion for a homogeneous polynomial to be square integrable on pseudoconvex balanced domains in dimension two (see Theorem~\ref{thm:criterion}). This theorem has many consequences as to the positive definiteness of the Bergman kernel and the positive definiteness of the Bergman metric is concerned (see Corollaries~\ref{corollary:positive-kernel},~\ref{corollary:bergman-metric},~\ref{corollary:all-polynomials}). In particular, with its help we may conclude the existence of domains of Siciak's type having completely different properties (see Theorem~\ref{theorem:siciak-class}).
 The fact that domains of Siciak's type admits so different possibilities appear as a surprise to us.

 A closely related problem to the study of the Bergman space is the theory of {\it the Bergman kernel} (we restrict our definition to that on the diagonal):
 \begin{equation}
 K_D(z):=\sup\{|f(z)|^2:f\in L_h^2(D),\;||f||\leq 1\},\; z\in D,
 \end{equation}
 {\it the Bergman pseudometric} (defined for $D$ such that $K_D(z)>0$, $z\in D$):
 \begin{equation}
 \beta_D^2(z;X):=\sum_{j,k=1}^n\frac{\partial^2\log K_D}{\partial z_j\bar\partial z_k}(z)X_j\bar X_k,\;z\in D,\; X\in\mathbb C^n,
 \end{equation}
 and {\it the Bergman distance } $b_D$.

Though the problems of positive definiteness, being an $L_h^2$-domain of holomorphy or Bergman complete in the class of bounded balanced pseudoconvex domains are either trivial or fully understood, the situation in the unbounded case is still unclear. We present a few results on these subjects. We think the methods and ideas presented in our paper may help the Reader to develop new methods to cope with the problems. Let us mention here that recently a lot of effort was invested in investigation of these problems in many classes of unbounded domains (see e. g. \cite{Chen 2013}, \cite{Chen 2015},  \cite{Chen-Kam-Ohs 2004}, \cite{Ahn-Gau-Kim 2015}, \cite{Pfl-Zwo 2005} or \cite{Pfl-Zwo 2007}). As we saw a problem of Wiegerinck drew a lot of effort recently and except for partial results already mentioned above the problem was repeated in a recent  survey on problems in the theory of several complex variables (\cite{For-Kim 2015}).

As already mentioned a special role in our considerations and motivation for understanding the phenomena in the class of unbounded balanced domains was played
by a special  class of (in some sense wild) domains. The example of that type was given by Siciak (see \cite{Sic 1982}) and it serves as a good candidate for counterexamples in many situations. Formally, the class $\mathcal D_S$ is the class of all  pseudoconvex balanced domains satisfying the following properties: $D=D_h$ where $h\not\equiv 0$ and $h^{-1}(0)$ is dense in $\mathbb C^n$. Such domains exist in arbitrary dimension (see \cite{Sic 1982}).

\section{Around balanced $L_h^2$-domains of holomorphy}

\subsection{Preliminaries}
Unless otherwise stated we shall work with
\begin{equation}
D:=D_h:=\{z\in\mathbb{C}^n:h(z)<1\}
\end{equation}
being a balanced domain, which means that $h:\mathbb C^n\to[0,\infty)$ is upper semicontinuous and homogeneous (i.\,e. $h(\lambda z)=|\lambda|h(z)$, $\lambda\in\mathbb C$, $z\in\mathbb C^n$). The fact that $D_h$ is pseudoconvex is equivalent to the plurisubharmonicity of $\log h$.
To avoid some trivialities we often assume that $D\neq\mathbb{C}^n$ (equivalently, $h\not\equiv 0$).

\bigskip

As already mentioned the main interest of our paper is devoted to unbounded balanced domains. A natural class of such domains, which may be a good starting point for exploring  properties of unbounded balanced domains, may be the class of \textit{ elementary balanced domains} by which we mean the family of domains defined by functions $h$ given by the formula:
\begin{equation}
h(z):=|A^1 z|^{t_1}\cdot\ldots\cdot|A^N z|^{t_N},\; z\in\mathbb C^n,
\end{equation}
  where $N$ is a positive integer, $A^1,\ldots,A^N:\mathbb C^n\to\mathbb C$ are non-zero linear mappings and the positive numbers $t_1,\ldots,t_N>0$ satisfy $t_1+\ldots+t_N=1$.

 In order to understand the situation in the class of unbounded balanced domains it may be useful to understand how the situation looks like for elementary balanced domains. To work with a general unbounded balanced domain often requires more sophisticated methods to work with.

Denote by
$\mathcal H^d$ the set of polynomials $P\in\mathbb C[z_1,\ldots,z_n]$ that are  homogeneous of degree $d$ or identically equal to $0$.

Define $\mathcal H^d(D):=\mathcal H^d\cap L^2(D)$.

One may find an orthonormal basis $\{P_{d,j}:d=0,1,\ldots, j\in \mathcal J_d\}$ of the space $L_h^2(D)$ where $P_{d,j}\in\mathcal H^d(D)$ (the sets of indices $\mathcal J_d$ are finite but may be also empty).

Then any $f\in \mathcal O(D)$  has the following expansion
\begin{equation}
f(z)=\sum_{d=0}^{\infty}P_d(z),
\end{equation}
where $P_d\in\mathcal H^d$ are uniquely determined and the convergence is locally uniform. If, additionally, $f\in L^2(D)$, then
$P_d=\sum_{j\in \mathcal J_d}\epsilon_{d,j}P_{d,j}$ and
\begin{equation}
||f||^2=\sum_{d=0}^{\infty}||P_d||^2=\sum_{d=0}^{\infty}\sum_{j\in\mathcal J_d}|\epsilon_{d,j}|^2;
\end{equation}
here $\|\cdot\|$ denotes the $L^2_h$-norm.

With any $f\in\mathcal O(D)$ we may define a homogeneous function
\begin{equation}
h_f(z):=\limsup_{d\to\infty}\root{d}\of{|P_d(z)|},\;z\in\mathbb C^n.
\end{equation}
Observe that
\begin{equation}
f(\lambda z)=\sum_{d=0}^{\infty}P_d(z)\lambda^d\in\mathbb C, \;z\in D, \lambda\in\mathbb D.
\end{equation}
Consequently, $h_f(z)\leq 1$, $z\in D$, and thus $h_f\leq h$.

One also gets the following equality
\begin{equation}
K_D(z)=\sum_{d=0}^{\infty}\sum_{j\in\mathcal J_d}|P_{d,j}(z)|^2,\; z\in D.
\end{equation}

Note that
\begin{equation}\label{eq:BK-expansion}
\infty>K_D(\lambda z)=\sum_{d=0}^{\infty}\left(\sum_{j\in\mathcal J_d}|P_{d,j}(z)|^2\right)|\lambda|^{2d},\;\lambda\in\mathbb D, z\in \mathbb C^n, h(z)\leq 1.
\end{equation}

Consequently, we may define the following homogeneous function (note that the above representation shows that the function defined below does not depend on the choice of the orthonormal basis $\{P_{d,j}\}$)
\begin{equation}
h_{BK}(z):=\limsup_{d\to\infty}\root{2d}\of{\sum_{j\in\mathcal J_d}|P_{d,j}(z)|^2},\; z\in \mathbb C^n.
\end{equation}
Note that $h_{BK}(z)\leq 1$, $z\in D$. Therefore, $h_{BK}\leq h$ on $\mathbb C^n$.

Let $f\in L_h^2(D)$ be represented as
\begin{equation}
f(z)=\sum_{d=0}^{\infty}P_d(z)=\sum_{d=0}^{\infty}\sum_{j\in \mathcal J_d}\epsilon_{d,j}P_{d,j}.
\end{equation}
 Then the Schwarz inequality gives
\begin{equation}
|P_d(z)|\leq\sqrt{\sum_{j\in\mathcal J_d}|\epsilon_{d,j}|^2}\sqrt{\sum_{j\in\mathcal J_d}|P_{d,j}(z)|^2}\leq
||f||\sqrt{\sum_{j\in\mathcal J_d}|P_{d,j}(z)|^2},
\end{equation}
from which we  get the following inequality
\begin{equation}
h_f\leq h_{BK}.
\end{equation}
Consequently, we have
\begin{equation}
h_f^*\leq h_{BK}^*\leq h,\; f\in L_h^2(D).
\end{equation}
Recall that $u^\ast$, $u$ a real valued function, means the upper semicontinuous regularization of $u$.

\subsection{$L_h^2$-domains of holomorphy of balanced domains}\label{subsection:domains-of-holomorphy}
Recall that the $L_h^2$-envelope of holomorphy of a balanced domain is univalent; even more, it is a balanced domain (see e.\,g. \cite{Jar-Pfl 2000}). Therefore, for a balanced domain $D_h$ there is another (pseudoconvex) balanced domain $D_{\tilde h}$ such that $D_h\subset D_{\tilde h}$ (equivalently, $\tilde h\leq h$), any $L_h^2$ function on $D$ extends to an element from $\mathcal O(D_{\tilde h})$, and the domain $D_{\tilde h}$ is the largest one with this property.

We give now the precise description of the $L_h^2$-envelope of holomorphy (in other words the description of the function $\tilde h$) of a balanced domain.

\begin{theorem} Let $D=D_h$ be a balanced domain in $\mathbb C^n$. Then its envelope of holomorphy is the domain $D_{h_{BK}^*}$.
\end{theorem}
\begin{proof}
Let $D_{\tilde h}$ be the $L_h^2$-envelope of holomorphy of $D$. First we show that $\tilde h\leq h_{BK}^*$.

Take any $f\in L_h^2(D)$ and let  $f=\sum_{d=0}^{\infty}\sum_{j\in\mathcal J_d}\epsilon_{d,j} P_{d,j}=\sum_{d=0}^{\infty}P_d$ with $\sum_{d=0}^{\infty}\sum_{j\in\mathcal J_d}|\epsilon_{d,j}|^2<\infty$. Then $|P_d|^2\leq \sum_{j\in\mathcal J_d}|\epsilon_{d,j}|^2\sum_{j\in\mathcal J_d}|P_{d,j}|^2$. We claim that the series defining $f$ converges locally uniformly on $\{h_{BK}^*<1\}$. In fact, take a compact $K\subset\{h_{BK}^*<1\}$. Then there is a $\theta\in(0,1)$ such that $h_{BK}^*(z)<1-2\theta$, $z\in K$. Consequently, the Hartogs lemma implies that $\root{d}\of{\sum_{j\in\mathcal J_{d,j}}|P_{d,j}(z)|^2}<1-\theta$, $z\in K$ and $d$ big enough. Therefore, the series $\sum_{d=0}^{\infty}P_d$ defines a holomorphic function on  $\{h_{BK}^*<1\}$ extending $f$. Thus $\tilde h\leq h_{BK}^\ast$.

 To finish the proof
suppose that there is a $z_0\in\mathbb C^n$ such that $\tilde h(z_0)<h_{BK}^*(z_0)$.
Then there is a $z_1$ such that
\begin{equation}
\tilde h(z_1)<\limsup_{d\to\infty}\root{2d}\of{\sum_{j\in\mathcal J_d}|P_{d,j}(z_1)|^2}=h_{BK}(z_1)=1\leq h(z_1).
\end{equation}
Consequently, for any $\epsilon>0$ the function
\begin{equation}
f:D\owns z\to \sum_{d=0}^{\infty}\sum_{j\in\mathcal J_d}\overline{P_{d,j}\left(\frac{z_1}{1+\epsilon}\right)}P_{d,j}(z)
\end{equation}
is from $L_h^2(D)$. To see this note that the series
$\sum_{d=0}^{\infty}\sum_{j\in\mathcal J_d}\left|P_{d,j}\left(\frac{z_1}{1+\epsilon}\right)\right|^2$ is convergent -- it is sufficient to see that $\limsup_{d\to\infty}\root{d}\of{\sum_{j\in\mathcal J_d}\left|P_{d,j}\left(\frac{z_1}{1+\epsilon}\right)\right|^2}=\frac{1}{(1+\epsilon)^2}<1$. The extendability of $f$ to the domain $D_{\tilde h}$ implies that the function
\begin{equation}
\lambda\to f(\lambda z_1)
\end{equation}
extends to a holomorphic function on $\{|\lambda|<\frac{1}{\tilde h(z_1)}\}$. But
\begin{equation}
f(\lambda z_1)=\sum_{d=0}^{\infty}\sum_{j\in\mathcal J_d}|P_{d,j}(z_1)|^2\frac{\lambda^d}{(1+\epsilon)^d},
\end{equation}
which would imply that $\frac{1}{1+\epsilon}=\limsup_{d\to\infty}\frac{\root{d}\of{\sum_{j\in\mathcal J_d}|P_{d,j}(z_1)|^2}}{1+\epsilon}\leq \tilde h(z_1)$, which is however impossible for $\epsilon>0$ small.
\end{proof}

Below we present a description of balanced $L_h^2$-domains of holomorphy.
\begin{theorem}\label{thm:main} Let $D\subset\mathbb C^n$ be a balanced domain. Then the following are equivalent
\begin{itemize}
\item $D$ is an $L_h^2$-domain of holomorphy,
\item $h(z)=h_f^*(z)$, $z\in \mathbb C^n$ for some function $f\in L_h^2(D)$,
\item $h(z)=h_{BK}^*(z)$, $z\in\mathbb C^n$.
\end{itemize}
\end{theorem}
  \begin{proof}
 The assumption that $D$ is an $L_h^2$-domain of holomorphy implies the existence of an $f\in L_h^2(D)$ which does not extend through any boundary point (see e. g. \cite{Jar-Pfl 2000}). Let us expand $f$ in homogeneous polynomials $f=\sum_{d=0}^{\infty}P_d$.  Suppose that there is a $z_0$ such that $h_f^*(z_0)=\left(\limsup_{d\to\infty}\root{d}\of{|P_d|}\right)^*(z_0)<h(z_0)$. Without loss of generality we may assume that $h_f^*(z_0)<1=h(z^0)$. Then the definition of the regularization and the Hartogs Lemma imply that there is a neighborhood $U$ of $z_0$ such that $|P_d(z)|<(1-\theta)^d$ for big $d$ and $z\in U$ with a suitable $\theta\in(0,1)$. Then the function
 $\sum_{d=0}^{\infty}P_d$ converges locally uniformly in $U$ and, consequently, $f$ extends holomorphically through the boundary point $z_0$ -- a contradiction.

 The equality $h\equiv h_f^*$ trivially implies the equality $h\equiv h_{BK}^*$.

 The lacking implication is, by the assumption of the equality $h\equiv h_{BK}^*$, a direct consequence of the previous theorem.

 \end{proof}

Recall that for  an arbitrary bounded pseudoconvex domain $D$ the fact that it is an $L_h^2$-domain of holomorphy is equivalently described by the boundary behaviour of its Bergman kernel (see \cite{Pfl-Zwo 2002}). We will extend one of these implications to the case of unbounded balanced domains.

\begin{corollary}
Let $D=D_h$ be balanced. Assume that
\begin{equation}
\limsup_{z\to z_0}K_D(z)=\infty,\; \text{ for any }z_0\in\partial D.
\end{equation}
Then $D$ is the $L_h^2$-domain of holomorphy.
\end{corollary}
\begin{proof}
 Suppose the opposite. Then there is a $z_0\in\partial D$ such that $h_{BK}^*(z_0)<1$. Then, as usual, there is a bounded neighborhood $U$ of $z_0$ such that $\sum_{j\in\mathcal J_d}|P_{d,j}(z)|^2\leq (1-\theta)^d$ for some $\theta\in(0,1)$, $d$ big enough and $z\in U$. Consequently, $\sum_{d=0}^{\infty}\sum_{j\in\mathcal J_d}|P_{d,j}(z)|^2<M$, $z\in U$, for some $M<\infty$. This expression, however, gives the value of $K_D(z)$ for $z\in U\cap D$, which contradicts the assumption.
  \end{proof}

  \begin{remark} Observe that for an arbitrary domain with univalent envelope of holomorphy the condition in the above corollary implies that $D$ is automatically pseudoconvex (see \cite{Nik-Pfl-Tho-Zwo 2009}).

  It remains an open question whether any balanced $L^2_h$-domain of holomorphy satisfies this condition as it is the case for bounded domains.

  \end{remark}

    It is intriguing that in many cases the upper regularization in Theorem~\ref{thm:main}  is
not necessary (at least in the case of the function $h_{BK}$). To be more concrete. Recall that pseudoconvex bounded balanced domains are always $L_h^2$-domains of holomorphy (see \cite{Jar-Pfl 1996}). Or even more, making use of the Theorem of Takegoshi-Ohsawa (\cite{Ohs-Tak 1987}) we get that for the bounded pseudoconvex balanced $D$ we have
(see \cite{Jar-Pfl-Zwo 2000})
\begin{equation}
\lim_{|\lambda|\to 1}K_D(\lambda z)=\infty,\; h(z)=1.
\end{equation}
This implies that for a fixed $z$ with $h(z)=1$ the series in (\ref{eq:BK-expansion}) as the power series of the variable $|\lambda|$ has the radius of convergence equal to one (and thus $h_{BK}(z)=1$). So we may rewrite the above result as follows:
\begin{proposition}\label{prop:bounded-case}
Let $D=D_h$ be a bounded pseudoconvex balanced domain, i.\,e. $h$ is logarithmically plurisubharmonic and $h^{-1}(0)=\{0\}$. Then we have the equality
\begin{equation}
h_{BK}(z)=h(z),\; z\in\mathbb C^n.
 \end{equation}
\end{proposition}

\begin{remark}
It is unclear how to get Proposition~\ref{prop:bounded-case} for arbitrary balanced $L_h^2$-domains of holomorphy (not necessarily bounded).

Let us once more underline that the above results except for their description of notions related with $L_h^2$-domain of holomorphy of balanced domains may also serve as results on approximation of homogeneous, logarithmically plurisubharmonic functions with the help of expressions of the form $\limsup_{d\to\infty}\root{d}\of{|P_d|}$ or $\limsup_{d\to\infty}\root{2d}\of{\sum_{j\in\mathcal J_d}|P_{d,j}(z)|^2}$ (sometimes after taking the upper semicontinuous regularization).
\end{remark}

\section{Bergman spaces in two-dimensional balanced domains}

Consider the following mapping
\begin{equation}
\Phi: \mathbb C^{n-1}\times(\mathbb C\setminus\{0\})\owns(z^{\prime},z_{n})\to\left(\frac{z^{\prime}}{z_n},z_n\right).
\end{equation}

The above mapping when restricted to $D_h$ allows to reduce the problem of describing  the dimension of the  space $L_h^2(D_h)$ to that of $L_h^2(G_{\phi})$ where
\begin{equation}
\phi(w^{\prime}):=\log h(w^{\prime},1),\; G_{\phi}:=\{w\in\mathbb{C}^n:|w_n|<\exp(-\phi(w^{\prime}))\}.
\end{equation}
In other words in order to find  $\dim L_h^2(D_h)$ it suffices to solve the same problem for a Hartogs domain with basis $\mathbb C^{n-1}$ and plurisubharmonic function $\phi:\mathbb C^{n-1}\to[-\infty,\infty)$ with the additional property $\limsup_{||w^{\prime}||\to\infty}(\phi(w^{\prime})-\log||w^{\prime}||)<\infty$ (obviously, here $\|\cdot\|$ is nothing than the Euclidean norm); we write that $\phi\in\mathcal L$.

\subsection{ Description of balanced domains in $\mathbb C^2$ with  trivial Bergman spaces}\label{subsection:trivial-dimension-two}

Below we study the question of Wiegerinck for the class of balanced pseudoconvex domains in $\mathbb C^2$ applying results from \cite{Juc 2012}.

Recall that
\begin{equation}
\limsup_{|\lambda|\to\infty}(\phi(\lambda)-\log|\lambda|)=\limsup_{|\lambda|\to\infty}\log h((1,\frac{1}{\lambda})\leq\log h(1,0)<\infty.
\end{equation}

\begin{theorem}\label{thm:description-dimension-two}
Let $D=D_h$ be a balanced pseudoconvex domain in $\mathbb C^2$. Then $L_h^2(D)=\{0\}$ or it is infinitely dimensional. Moreover, $L_h^2(D)=\{0\}$ iff $h(z)=|Az|^t|Bz|^{1-t}$ or $h(z)=|Az|$ or $h(z)=0$, where $A,B$ are non trivial linear mappings $\mathbb C^2\to\mathbb C$ and $t\in(0,1)$.
  \end{theorem}
 \begin{proof} Without loss of generality assume that $D\neq\mathbb C^2$. It follows from Theorem 4.1 in \cite{Juc 2012} that $L_h^2(D)=\{0\}$ iff
$\phi(\lambda)=\sum_j\alpha_j|\lambda-a_j|+g(\lambda)$, where $\alpha_j>0$, $g$ is harmonic on $\mathbb C$, and some additional condition on the $\alpha_j$'s is satisfied. And in the case $L_h^2(D)$ is not trivial it is infinitely dimensional.

We know (see e. g. \cite{Juc 2012}) that $\triangle \phi\leq 2\pi$. In particular, $\sum_j\alpha_j\leq 1$. It also follows from the description of the triviality of $L_h^2(G_{\phi})$ in Jucha's paper (conditions on $\alpha_j$) that at most two $\alpha_j$'s are positive.

Then
\begin{equation}
\limsup_{|\lambda|\to\infty}\left(\sum_j\alpha_j\log\left|1-\frac{a_j}{\lambda}\right|+\left(\sum_j\alpha_j-1\right)\log|\lambda|+g(\lambda)\right)<\infty
\end{equation}
from which we conclude that
\begin{equation}
\limsup_{|\lambda|\to\infty}\left(\left(\sum_j\alpha_j-1\right)\log|\lambda|+g(\lambda)\right)<\infty.
\end{equation}
The above function is from $\mathcal L$, $g$ is harmonic on $\mathbb C$, $\sum_j\alpha_j\leq 1$ (and the sum is taken over a finite set). Thus,  one easily gets from the standard properties of harmonic functions that $g\equiv const$ and $\sum_j\alpha_j=1$, which easily finishes the proof.
  \end{proof}

\begin{remark} The above description of balanced domains with trivial Bergman space may be formulated as follows: $L_h^2(D)=\{0\}$ iff
$D$ is linearly equivalent to $\mathbb C^2$, $\mathbb C\times\mathbb D$ or $\{z\in\mathbb C^2:|z_1|^t|z_2|^{1-t}<1\}$ for some $t\in(0,1)$.
\end{remark}

 \begin{remark} The above result leaves many questions open. For instance does the dichotomy on the dimension of $L_h^2(D)$ remain true for pseudoconvex balanced domains in higher dimension?

 What is the description of pseudoconvex balanced domains with a trivial Bergman space in higher dimension? Do they have to be elementary balanced domains as it is the case in dimension two?

 Is every pseudoconvex balanced domain of holomorphy with non-trivial Bergman space an $L_h^2$-domain of holomorphy (as it is the case for the bounded domains)?

It follows from the above result that domains from the class $\mathcal D_S$ (in dimension two) always admit an infinitely dimensional Bergman space. Are these domains also $L_h^2$-domains of holomorphy?
\end{remark}

\begin{remark}
 Recall that the main result in \cite{Gal-Harz-Her 2016} assures the infinite dimensionality of $L_h^2(D)$ in the case the core of the domain is different from $D$. Note that domains from the class $\mathcal D_S$ do not satisfy the equality $\mathfrak c'(D)=D$. Therefore, the infinite dimensionality of $L^2_h(D)$ for these domains cannot be concluded from \cite{Gal-Harz-Her 2016}.

 \end{remark}

   \subsection{Square integrable homogeneous polynomials}

 Let $h:\mathbb C^2\to[0,\infty)$ be a logarithmically plurisubharmonic, homogeneous function with $h\not\equiv 0$.

For $[v]\in\mathbb P^1$ with $v_2\neq 0$ (the case when $v_1\neq 0$ is defined appropriately) define
\begin{equation}
\nu(\log h,[v]):=\nu(\log h(\cdot,v_2),v_1),
\end{equation}
where $\nu(u,\cdot)$ denotes the Lelong number of a subharmonic function $u$. Note that the notion is well-defined (it does not depend on the choice of a representant of $[v]$ and the number of the chosen variable).

Recall that for a subharmonic function $u:D\to[-\infty,\infty)$, $z_0\in D$ ($D\subset\mathbb C$ a domain)  \textit{ the Lelong number is given as} (see e. g. \cite{Ran 1995}):
\begin{equation}
\nu(u,z_0):=\lim_{r\to 0}\frac{\frac{1}{2\pi}\int_0^{2\pi}u(z_0+re^{it})dt}{\log r}=\lim_{r\to 0}\frac{\max_{|\lambda|=r}u(z_0+\lambda)}{\log r}=\triangle u(\{z_0\}).
\end{equation}

 We shall need the following properties of $\nu$ that will follow directly from that of the standard Lelong number:

\begin{itemize}
 \item $\sum_{[v]\in\mathbb P^1}\nu(\log h,[v])\leq 1$, $[v]\in\mathbb P^1$;\\
 \item  $\nu(t\log h_1+(1-t)\log h_2,\cdot)=t\nu(\log h_1,\cdot)+(1-t)\nu(\log h_2,\cdot)$, $t\in[0,1]$.\\
           \end{itemize}

           For an element $[v]\in\mathbb P^1$ (with $v_2\neq 0$) and an $\epsilon>0$ we introduce the conic neighborhood of $[v]$ in $\mathbb C^2$ as follows (in the case $v_2=0$ we may define it analoguously with $v_1$ chosen):
           \begin{equation}
           U([v],\epsilon):=\left\{\frac{\mu}{v_2}(v_1,\lambda):|\lambda-v_1|<\epsilon|v_2|,\lambda, \mu\in\mathbb C\right\}.
           \end{equation}

For $0\not\equiv Q\in\mathcal H^d$ we calculate (assuming that $v_2\neq 0$):
\begin{multline}
\int_{\{h<1\}\cap U([v],\epsilon)}|Q(z)|^2d\mathcal L^4(z)=\\
\int_{\{w_1:|w_1-v_1|<\epsilon|v_2|\}}\left(\int_{\{w_2:|w_2|<\exp(-\phi(w_1))\}}|Q(w_1,1)|^2|w_2|^{2d+2}d\mathcal L^2(w_2)\right)d\mathcal L^2(w_1)=\\
\frac{\pi}{d+2}\int_{\{w_1:|w_1-v_1|<\epsilon|v_2|\}}|Q(w_1,1)|^2\exp\left(-2(d+2)\phi(w_1)\right)d\mathcal L^2(w_1).
\end{multline}

It follows from  properties of the Lelong numbers (see Corollary 2.4(a) in [Juc 2012]) that the last expression is finite (for sufficiently
small $\epsilon>0$) iff
\begin{equation}
m(Q,[v])>\nu(\log h,[v])(deg Q+2)-1.
\end{equation}
Here $m(Q,[v])$ denotes the multiplicity of the zero of the homogeneous polynomial
by which we mean the multiplicity of the zero of the polynomial $Q(v_1,\lambda)$ at $\lambda=v_2$ (if $v_1\neq 0$) or the multiplicity of the polynomial $Q(\lambda,v_2)$ at $\lambda=v_1$ (if $v_2\neq 0$).

Making use of the standard compactness argument we get the following.

\begin{theorem}\label{thm:criterion} Fix $n=2$.
Let $Q\in\mathcal H^d$ and let $h:\mathbb C^2\to[0,\infty)$ be a logarithmically plurisubharmonic homogeneous function. Then $Q\in L_h^2(D_h)$ iff
 \begin{equation}
 m(Q,[v])>\nu(\log h,[v])(deg Q+2)-1\; \text{ for any $[v]\in\mathbb P^1$}.
 \end{equation}
\end{theorem}

Now, the above result has many consequences.

 \begin{corollary}\label{corollary:positive-kernel} For the balanced pseudoconvex domain $D_h\subset\mathbb C^2$ the following are equivalent:
\begin{itemize}
\item $K_D(0)>0$;\\
\item $K_D(z)>0$, $z\in D$;\\
\item $\mathcal L^4(D_h)<\infty$;\\
\item $\nu(\log h,[v])<1/2$, $[v]\in\mathbb P^1$.
\end{itemize}
\end{corollary}
 \begin{proof}
Recall that for the balanced pseudoconvex domain we get $K_D(0)=\frac{1}{\mathcal L^4(D)}$, which easily gives the equivalence of the first three conditions equivalent.  Theorem~\ref{thm:criterion} applied to $Q\equiv 1$  gives the equivalence of its integrability to the last condition.
\end{proof}

 Similarly, we get the following.

\begin{corollary}\label{corollary:bergman-metric} For the balanced pseudoconvex domain $D_h\subset\mathbb C^2$ the following are equivalent:
\begin{itemize}
\item $D_h$ admits the Bergman metric;\\
\item the functions $z_1,z_2\in L_h^2(D_h)$;\\
\item $\nu(\log h,[v])<1/3$, $[v]\in\mathbb P^1$.
\end{itemize}
 \end{corollary}

   \begin{proof}
  The equivalence of the second and third condition follows from Theorem~\ref{thm:criterion} (applied to the functions $z_1$ and $z_2$). It $D_h$ admits the Bergman metric, then the positive definiteness of the Bergman metric at $0$ easily implies that $z_1,z_2$ are square integrable. On the other hand assuming the third property we get from the previous corollary the positive definiteness of the Bergman kernel. The fact that in this the functions $z_1,z_2$ are square integrable follows from Theorem~\ref{thm:criterion} which finishes the proof.
      \end{proof}

  \begin{example} To illustrate the results from above we give the following three simple examples:

  1) $h(z,w):=|zw|^{1/2}$. Then $K_{D_h}(0)=0$.

  2) $h(z,w):=|z(z-w)w|^{1/3}$. Then $K_{D_h}(0)>0$, but $D_h$ admits no Bergman metric.

  3) $h(z,w):=|z(z-w)(z+w)w|^{1/4}$. Then $D_h$ admits a Bergman metric.
  \end{example}

  \begin{corollary}\label{corollary:all-polynomials}
  Let $D_h\subset\mathbb C^2$ be a balanced pseudoconvex domain. Then $L_h^2(D_h)$ contains all the polynomials iff $\nu(\log h,[v])=0$ for all $[v]\in\mathbb P^1$.
  \end{corollary}

   \begin{remark} Recall that there are balanced (in fact even Reinhardt) domains that are not bounded but that satisfy the assumption of the previous corollary. In fact, we may take for instance
   \begin{equation}\label{remark:example}
   D:=\left\{z\in\mathbb C^2:|z_2|<1 \text{ and } |z_2|<\exp(-(\log|z_1|)^2) \text{ for $|z_1|\geq 1$}\right\}.
   \end{equation}
The fact that all the monomials belong to $L_h^2(D)$ may be verified either by the direct estimates of the appropriate integrals or by applying the description of square integrable monomials in pseudoconvex Reinhardt domains that may be found in \cite{Zwo 1999} or \cite{Zwo 2000}.
   \end{remark}

   \bigskip

 It turns out that there are domains from $\mathcal D_S$ having quite different properties as its connection with the Bergman theory is concerned. First recall that we have already seen that all such domains have their $L_h^2$ space infinitely dimensional. Additionally, we have the following
 \begin{theorem}\label{theorem:siciak-class}
 \begin{itemize}
\item There is a $D\in\mathcal D_S$ such that its four-dimensional volume $\mathcal L^4(D)=\infty$.\\
 \item There is a $D\in\mathcal D_S$ such that $K_D(z)>0$, $z\in D$ (equivalently, $\mathcal L^4(D)<\infty$) but $D$ does not admit the Bergman metric.\\
 \item There is a $D\in\mathcal D_S$ such that $D$ admits the Bergman metric.
 \end{itemize}
 \end{theorem}

 \begin{proof}
  Let $D=D_h$ be any domain from $\mathcal D_S$. For a positive integer $n$ choose $[(1,a_1)],\ldots,[(1,a_n)]\in\mathbb P^1$ such that $\nu(\log h,[(1,a_j)])=0$, $j=1,\ldots,n$ (it is always possible!). Then for any $t\in(0,1)$ consider the functions defined as follows
  \begin{equation}
  t \log h(z)+\frac{1-t}{n}\sum_{j=1}^n\log|z_1-a_jz_2|.
  \end{equation}

  It is now a straightforward consequence of the above corollaries that independently of the property the starting function $h$ has one may, manipulating with $n$ and $t$, get the function delivering  domains with in all the claimed three points from above.
 \end{proof}

\begin{remark}
Examples that were studied in \cite{Pfl-Zwo 2005} show that it is a kind of the Lelong number which may be responsible for the fact that the balanced domain is Bergman complete. More precisely, the following description of Bergman complete balanced domains in $\mathbb C^2$ may be correct:

$D=D_h$ is Bergman complete iff  $\nu(\log,[v])=0$ for any $[v]\in\mathbb P^1$.

Recall that all  bounded pseudoconvex balanced domains are Bergman complete (see \cite{Jar-Pfl-Zwo 2000}) as well as the domain defined in (\ref{remark:example}) (see \cite{Pfl-Zwo 2005}).

 As we have seen the domains from the class $\mathcal D_S$ may have quite different properties. But does there exist a two-dimensional domain $D=D_h\in\mathcal D_S$ such that $\nu(\log,[v])=0$ for any $[v]\in\mathbb P^1$?
 \end{remark}

\end{document}